\documentclass{article}
\usepackage{amssymb, amsmath, mathrsfs, amsthm}
\usepackage[nobysame]{amsrefs}
\usepackage[notcite,notref]{}
\usepackage{ulem}
\usepackage{graphicx}
\usepackage{color}
\usepackage[margin=1in]{geometry}
\usepackage{float, caption, subcaption}
\usepackage{algorithm}
\usepackage[noend]{algpseudocode}
\usepackage{hyperref,verbatim}

\usepackage{tikz}
\usetikzlibrary{shapes,decorations,calc,fit,positioning}
\tikzset{
    auto,node distance =1 cm and 1 cm,semithick,
    state/.style ={ellipse, draw, minimum width = 0.7 cm}
}

\usepackage{tikz}
\usetikzlibrary{shapes,decorations,calc,fit,positioning}
\tikzset{
    auto,node distance =1 cm and 1 cm,semithick,
    state/.style ={ellipse, draw, minimum width = 0.7 cm}
}

\newtheorem{theorem}{Theorem}

\theoremstyle{definition}
\newtheorem{definition}{Definition}
\newtheorem{remark}{Remark}
\newtheorem{example}{Example}
\newtheorem{lemma}{Lemma}
\newtheorem{corollary}{Corollary}
\newtheorem{proposition}{Proposition}

\newcommand{\fq}{\mathbb{F}_q}

\newcommand{\F}{\mathbb{F}}

\begin{document}

\title{Positive-Definite Matrices over Finite Fields}
\markright{Positive-Definite Matrices over Finite Fields}

\author{Joshua Cooper, Erin Hanna, and Hays Whitlatch}

\date{February 07, 2022}

\maketitle

\begin{abstract}
The study of positive-definite matrices has focused on Hermitian matrices, that is, square matrices with complex (or real) entries that are equal to their own conjugate transposes.  In the classical setting, positive-definite matrices enjoy a multitude of equivalent definitions and properties.  In this paper, we investigate when a square, symmetric matrix with entries coming from a finite field can be called ``positive-definite'' and discuss which of the classical equivalences and implications carry over.
\end{abstract}

\section{Introduction}
Our goal in this paper is to introduce and investigate the concept of positive-definiteness over finite fields.  What does it mean to be positive in a finite field?  Are exactly half of the non-zero elements positive? Do the equivalences afforded to Hermitian positive-definite matrices carry over to finite fields? And naturally, why is this topic interesting, and how did this come up?  We encourage the reader to pause and think about these questions before reading our proposed framework for addressing them.

The theory of positive-definiteness is vast and reaches throughout many branches of mathematics. Positive-definite matrices are used in optimization algorithms, to study convexity with multi-variable functions, in the construction of some types of linear regression models, and in principal component analysis.  Positive-definite Hermitian matrices admit a ``Cholesky decomposition'', that is, the matrix can be expressed as the product of a lower triangular matrix and its conjugate transpose.  This decomposition can be used to efficiently solve a system of linear equations or to quickly acquire numerical solutions via repeated random sampling in a computational algorithm such as Monte Carlo simulation.

Our interest in positive-definite matrices emerged from a very different application in which the success of a graph theoretical operation (called ``pressing'') corresponds to the existence of a Cholesky decomposition for its adjacency matrix (over the finite field $\mathbb{F}_2$).  This correspondence was first shown in \cite{CooperDavis} and was further interpreted in \cite{CooperWhitlatch}.  The topic originated via an application in bioinformatics (see \cite{bixby2015proving}, \cite{hannenhalli1999transforming}), where pressing sequences correspond to sortings by reversal of DNA sequences.  We elaborate briefly on the connection with matrix algebra below as it is relevant to the discussion in Section \ref{pressingsection}.

\begin{definition}
A \textbf{bicolored graph} $G = (G,c)$ is a simple graph $G$ with $c: V(G) \rightarrow \{\text{blue}, \text{white}\}$ which assigns a color to each vertex.\footnote{Some authors use black and white instead.}  The complement of blue is white and the complement of white is blue. For 
$v \in V(G)$,  \textbf{pressing} a blue vertex $v$ is the operation of transforming $(G,c)$ to $(G',c')$, a new bicolored graph in which $G[N_G(v) \cup \{v\}]$ is complemented.  That is, $V(G) = V(G')$ and $$E(G')= E(G) \triangle \begin{pmatrix}
N_G(v) \cup \{v\} \\
2
\end{pmatrix}$$
where $\triangle$ denotes symmetric difference, $N_G(v)$ is the neighborhood of $v$ in $G$, $c'(w)$ is the complement of $c(w)$ for $w \in N_G(v) \cup \{v\}$, and $c'(w) = c(w)$ otherwise.  This definition is illustrated in Example \ref{pressingexample} where one can see that pressing a blue vertex changes its color to white and flips the colors of its neighbors, isolates the pressed vertex, and complements the set of edges induced by its neighbors.
\end{definition}

A sequence of presses is referred to as a \textbf{pressing sequence}.  Since a pressed vertex becomes isolated, it may not be pressed again, and will not be affected by future presses.  Thus, every pressing sequence is finite in length.  If the end result of a pressing sequence is the empty, edgeless, colorless graph then we say that it was a \textbf{successful pressing sequence}.  As it turns out, a pressing sequence is successful exactly when the graph's adjacency matrix (with row and column order dictated by the pressing sequence) has a Cholesky decomposition over $\mathbb{F}_2$.  In fact, the Cholesky decomposition gives the instructions for pressing, revealing at each press which vertices would be affected.

\begin{example}\label{pressingexample}
Below is a bicolored graph with successful pressing sequence 1,2,3,4 (blue vertices are shown in gray). 

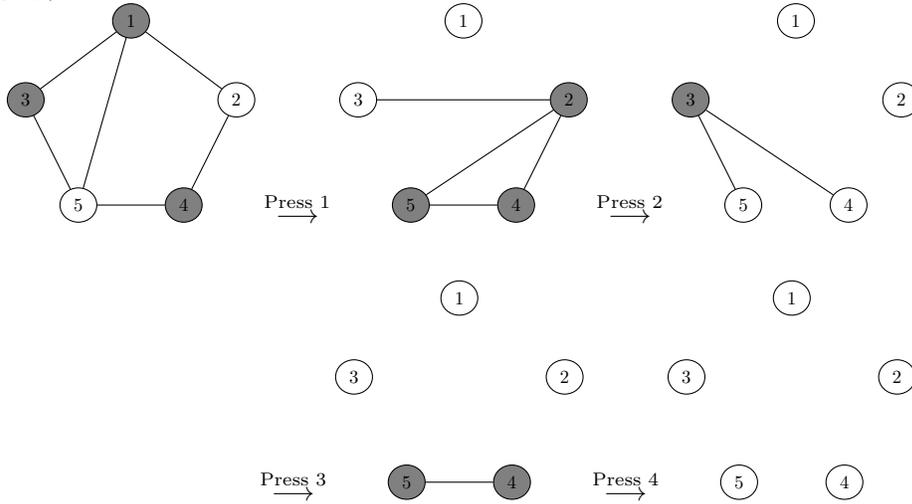
\begin{tikzpicture}[scale=.7, transform shape]

\node[state, fill=gray] (1) at (0, -.5) {1};
\node[state, fill=white] (2) at (2,-2) {2};
\node[state, fill=gray] (4) at (1,-4) {4};
\node[state, fill=white] (5) at (-1,-4) {5};
\node[state, fill=gray] (3) at (-2,-2) {3};

\draw (1) -- (2);
\draw (1) -- (5);
\draw (1) -- (3);
\draw (2) -- (4);
\draw (4) -- (5);
\draw (5) -- (3);
\end{tikzpicture}
$\overset{\text{Press 1}} \longrightarrow$
\begin{tikzpicture}[scale=.7, transform shape]

\node[state, fill=white] (1) at (0, -.5) {1};
\node[state, fill=gray] (2) at (2,-2) {2};
\node[state, fill=gray] (4) at (1,-4) {4};
\node[state, fill=gray] (5) at (-1,-4) {5};
\node[state, fill=white] (3) at (-2,-2) {3};

\draw (2) -- (3);
\draw (2) -- (4);
\draw (2) -- (5);
\draw (4) -- (5);
\end{tikzpicture}
$\overset{\text{Press 2}} \longrightarrow$
\begin{tikzpicture}[scale=.7, transform shape]

\node[state, fill=white] (1) at (0, -.5) {1};
\node[state, fill=white] (2) at (2,-2) {2};
\node[state, fill=white] (4) at (1,-4) {4};
\node[state, fill=white] (5) at (-1,-4) {5};
\node[state, fill=gray] (3) at (-2,-2) {3};

\draw (4) -- (3);
\draw (5) -- (3);
\end{tikzpicture}
\vspace{4mm}

\begin{center}
$\overset{\text{Press 3}} \longrightarrow$
\begin{tikzpicture}[scale=.7, transform shape]

\node[state, fill=white] (1) at (0, -.5) {1};
\node[state, fill=white] (2) at (2,-2) {2};
\node[state, fill=gray] (4) at (1,-4) {4};
\node[state, fill=gray] (5) at (-1,-4) {5};
\node[state, fill=white] (3) at (-2,-2) {3};

\draw (4) -- (5);
\end{tikzpicture}
$\overset{\text{Press 4}} \longrightarrow$
\begin{tikzpicture}[scale=.7, transform shape]

\node[state, fill=white] (1) at (0, -.5) {1};
\node[state, fill=white] (2) at (2,-2) {2};
\node[state, fill=white] (4) at (1,-4) {4};
\node[state, fill=white] (5) at (-1,-4) {5};
\node[state, fill=white] (3) at (-2,-2) {3};

\end{tikzpicture}
\end{center}
The adjacency matrix of this graph (with 1's on the diagonal when the vertex is blue) and the Cholesky decomposition corresponding to this pressing sequence (with rows and columns labelled by 1, 2, 3, 4, 5) is 
$$\left[\begin{array}{ccccc}
     1 & 1 & 1 & 0 & 1 \\
     1 & 0 & 0 & 1 & 0 \\
     1 & 0 & 1 & 0 & 1 \\
     0 & 1 & 0 & 1 & 1 \\
     1 & 0 & 1 & 1 & 0 \\
\end{array}\right]=\left[\begin{array}{ccccc}
     1 & 1 & 1 & 0 & 1 \\
     0 & 1 & 1 & 1 & 1 \\
     0 & 0 & 1 & 1 & 1 \\
     0 & 0 & 0 & 1 & 1 \\
     0 & 0 & 0 & 0 & 0 \\
\end{array}\right]^T\left[\begin{array}{ccccc}
     1 & 1 & 1 & 0 & 1 \\
     0 & 1 & 1 & 1 & 1 \\
     0 & 0 & 1 & 1 & 1 \\
     0 & 0 & 0 & 1 & 1 \\
     0 & 0 & 0 & 0 & 0 \\
\end{array}\right]$$
Observe that in the upper diagonal matrix, the $j^\textrm{th}$ entry of the $i^\textrm{th}$ row is $1$ exactly when pressing $i$ changed the state of $j$.
\end{example}

In the introduction we briefly discussed that Hermitian positive-definite matrices enjoy a multitude of theoretical and computational interpretations.  As such, there are several equivalent definitions for a positive-definite Hermitian matrix and the choice of which definition to use is often determined by the desired application.  We will see in Proposition \ref{stack} that not all of these definitions make sense over finite fields.

\begin{definition}(\cite{johnson1985matrix})
A $n\times n$ Hermitian matrix $A$ is said to be positive-definite if any (and hence all) of the following hold:
\begin{enumerate}
    \item $z^*Az >0$ for all non-zero column vectors $z\in \mathbb{C}^n$;
    \item $A$ has positive eigenvalues;
    \item The associated sesquilinear form is an inner product;
    \item $A$ is the Gram matrix of linearly independent vectors;
    \item All leading principal minors of $A$ are positive;
    \item $A$ has a unique Cholesky decomposition.
\end{enumerate}
\end{definition}

In addition, Hermitian positive-definite matrices enjoy numerous useful properties, only some of which will carry over to finite-fields. If $A$ and $B$ are positive-definite $n\times n$ Hermitian matrices, then the following statements hold (\cite{johnson1985matrix}):
\begin{enumerate}
    \item $A$ is invertible;
    \item $A^{-1}$ is positive-definite;
    \item $rA$  is positive-definite ($r>0$ a real number);
    \item $ABA$ and $BAB$ are positive-definite;
    \item if $AB=BA$ then $AB$ is positive-definite;
    \item every principal submatrix of $A$ is positive-definite;
    \item the Hadamard product $A\circ B$ and the Kronecker product $A \otimes B$ are positive-definite;
    \item the Frobenius product $A : B\geq 0$.
\end{enumerate}

\section{Defining Positive-Definiteness in a Finite Field}

In this section we will introduce the notion of positive-definiteness over a finite field.  In Section 1 we presented several equivalent definitions that work over $\mathbb{C}$.  Among these, the positivity of $z^*Az$ for all non-zero column vectors $z$ is perhaps the most commonly used definition for positive-definiteness (see for example \cite{johnson1985matrix}).  Over a finite field, this definition fails to translate meaningfully for a few reasons.  First,  there is the ambiguity of the statement $z^*Az >0$, which should generally mean something different than $z^*Az\neq 0$.  We will resolve this issue in Definition \ref{define_positive}.
However, the following proposition demonstrates that, even if we resolve the previous issue, the definition will not apply to any finite field.

\begin{proposition}\label{stack} Suppose $\mathbb{F}$ is a finite field, and $A$ is an $n \times n$ matrix over $\mathbb{F}$ for $n \geq 3$.   Define $Q:\mathbb{F}^n\rightarrow \mathbb{F}$ by $Q(x)=x^TAx$.  Then there exists a non-zero vector $v$ so that $Q(v) = 0$.
\end{proposition}
\begin{proof}
Chevalley's Theorem (\cite{chevalley1935demonstration}) states that, for a collection of polynomials $\{f_1,\ldots,f_m\} \subseteq \mathbb{F}[X_1,\ldots,X_n]$ with $n > \sum_{j=1}^m \deg f_i$, if $f_i(0,\ldots,0) = 0$ for all $i$, then there exists a nonzero $v$ so that $f_i(v) = 0$ for all $i$.  Take $m=1$ and $f_1(x) = Q(x)$, where we are treating the coordinates of $x$ as the variables $X_i$ on the left-hand side.  Then, since $n \geq 3 > 2 = \deg f_1$, and $Q(0,\ldots,0)=0$, the proposition follows.
\end{proof}

The previous proposition demonstrates that it is not possible to extend all of the definitions (and implications) of Hermitian positive-definiteness to finite fields.  Before proceeding, we must first resolve the question of what it means to be positive in a finite field.

\begin{definition}\label{define_positive}
For $x\in \mathbb{F}$, we say $x$ is \textbf{positive} if $x=\mu^2$ for some $\mu\neq 0 \in \mathbb{F}$ and we say $\mu$ is a square root of $x$.  If $\mu$ is also positive then we say $\mu$ is a positive square root of $x$.  
\end{definition}

\begin{definition}
Define a field $\mathbb{F}$ to be a \textbf{definite field} if each positive element has a positive square root.  If the field $\mathbb{F}$ is finite then we will refer to it as a \textbf{finite definite field}.
\end{definition}
\begin{remark}
Observe that $2$ is not positive in $\mathbb{F}_3$ but $2$ is positive in $\mathbb{F}_7$.  We strongly considered referring to the positive elements of $\mathbb{F}_q$ as ``$q$-positive'' elements and pronouncing $q-$positive as ``quositive''.  
\end{remark}
\begin{example}
$\mathbb{R}$ is a definite field:  Let $x=\mu^2$ for some $\mu\neq 0 \in \mathbb{R}$ and observe that   $x=(|\mu|)^2$ as well and $|\mu|\neq 0 \in \mathbb{R}$.  However $\sqrt{|\mu|}\neq 0 \in \mathbb{R}$ satisfies that $|\mu|=(\sqrt{|\mu|})^2$ so $x$ is positive and $|\mu|$ is a positive square root of $x$.
\end{example}
\begin{example}
$\mathbb{F}_3$ is a definite field: The non-zero elements are $1$ and $2$.  Since $1^2=1$ and $2^2=1$ we see that $1$ is positive and $2$ is not positive (it has no square root).  So if $x$ is positive then $x=1$, setting $\mu=1$ shows that $x$ has a positive square root. 
\end{example}
The following example demonstrates that not all finite fields are definite.
\begin{example}
Consider $f:\mathbb{F}_5\rightarrow\mathbb{F}_5$ by $f(x)=x^2$.  Then $f(0)=0$, $f(1)=f(4)=1$, and $f(2)=f(3)=4$.  Then, $1$ and $4$ are the positive elements of $\mathbb{F}_5$.  Observe that $1$ has two positive square roots (itself and $4$), however both of the square roots of $4$ are non-positive.  Therefore, $\mathbb{F}_5$ is not a definite field.
\end{example}

\begin{definition}{(\cite{hardy1979introduction})}\label{defineLegendre} The \textbf{Legendre symbol} $\left(\frac{a}{p}\right)$ for an integer $a$ and an odd prime $p$ is defined as
$$
\left(\frac{a}{p}\right)=\begin{cases}
1&\textrm{if there exists a nonzero $x$ such that }x^2\equiv a \pmod{p}\\ 
0&\textrm{if }a\equiv 0 \pmod{p}\\
-1&\textrm{otherwise }\\
\end{cases}
$$
\end{definition}
\begin{lemma}[\cite{hardy1979introduction}]\label{legendre}
Let $p$ be an odd prime.  The quadratic character of $-1$ modulo $p$ depends only on whether $p$ is $1$ or $3$ modulo $4$.  That is, $$\left(\frac{-1}{p}\right)=\begin{cases}
1&\textrm{if }p\equiv 1 \pmod{4}\\ -1&\textrm{if }p\equiv 3 \pmod{4}\\
\end{cases}$$
\end{lemma}
Lemma \ref{legendre}  can be used to show that $\mathbb{F}_p$ ($p$ an odd prime) is definite if and only if $p\equiv 3\pmod{4}$.  Suppose first that $p\equiv 3\pmod{4}$ and let $a$ be a positive element of $\mathbb{F}_p$.  Then for some $\mu\neq 0\in \mathbb{F}_p$ we have $\mu^2=a$ and therefore $(-\mu)^2=a$ as well.  Since the Legendre symbol is multiplicative $$\left(\frac{\mu}{p}\right)=\left(\frac{-1}{p}\right)\left(\frac{-\mu}{p}\right)=-\left(\frac{-\mu}{p}\right)$$ so either $\left(\frac{\mu}{p}\right)=1$ or $\left(\frac{-\mu}{p}\right)=1$ which implies that $a$ has a positive square root and therefore $\mathbb{F}_p$ is a definite field.  Suppose now that $p\equiv 1\pmod{4}$.   Since squaring (non-zero elements) is a two-to-one function, it is not possible for every element to be positive.  Let $b$ be a non-positive element of $\mathbb{F}_p\setminus\{0\}$ and let $a=b^2$ (so $a$ is positive).  Now  $$\left(\frac{b}{p}\right)=\left(\frac{-1}{p}\right)\left(\frac{-b}{p}\right)=\left(\frac{-b}{p}\right)$$ so $-b$ is a non-positive as well.  It follows (by the Fundamental Theorem of Algebra) that $a$ does not have a positive square root and therefore $\mathbb{F}_p$ is not a definite field.  Unfortunately this simple argument doesn't extend to $\mathbb{F}_q$ where $q$ is a proper prime power.  To see this one simply needs to observe that the non-prime analog of the Legendre symbol would have (weaker) implications for $\mathbb{Z}/q\mathbb{Z}$ but not for $\mathbb{F}_q$.  
\begin{theorem}\label{finite_definite_fields}
Let $\fq$ be a finite field. $\fq$ is a finite definite field if and only if  $\fq$ has characteristic two or $q=p^k$ where $k$ is a positive odd integer and $p\equiv 3 \pmod{4}$.
\end{theorem}
\begin{proof}
If $\fq$ has characteristic $2$ then the Frobenius map $f:\fq\rightarrow\fq$ given by $f(x)=x^2$ is an automorphism and therefore every element of $\fq$ has a square root\footnote{Indeed, this argument shows that every perfect field of characteristic $2$ is definite.}.  Suppose now that $q=p^k$ where $p$ is an odd prime and $k$ is a positive integer.  Let $\mathcal{P}=\mathcal{P}(\fq)$ be the set of positive elements in  $\mathbb{F}_q$ and $\mathbb{F}_q^*=\left(\mathbb{F}_q\setminus\{0\}, \times \right)$ be the multiplicative group of $\mathbb{F}_q$.  Since $\mathbb{F}_q^*$ is cyclic, there exists an element $a$, such that the $q-1$ elements of $\mathbb{F}_q^*$ are $a, a^2, \ldots , a^{q-2}, a^{q-1} = 1$.  Thus, $\mathcal{P}=\{(a)^2, (a^2)^2, \ldots , (a^{q-2})^2, (a^{q-1})^2\}=\{a^2, a^{2\cdot 2}, \ldots , a^{2\cdot(q-2)}, a^{2\cdot(q-1)}\}$.  However $a^{q-1+k}=a^{q-1}a^{k}=a^k$ so  $$\mathcal{P}=\left\{a^2, a^4, \ldots, a^{2\left(\frac{q-3}{2}\right)}, a^{2\left(\frac{q-1}{2}\right)} \right\}=\left\{a^{2t}\mid 1\leq t\leq \frac{q-1}{2}\right\}.$$  
Recall that $\fq^*$ is a definite field if for each $y\in \mathcal{P}$ there is a $r\in \mathcal{P}$ such that $r^2=y$.  That is, if and only if for all $1\leq t\leq \frac{q-1}{2}$  there is an integer $s$ such that $1\leq s\leq \frac{q-1}{2}$ and $r^2=(a^{2s})^2=a^{2t}=y$.  This occurs if and only if $2s\equiv t\pmod{\frac{q-1}{2}}$ has a solution for all $1\leq t\leq \frac{q-1}{2}$.  That is, if and only if $2s\equiv 1\pmod{\frac{q-1}{2}}$ has a solution.  Observe that $$q\equiv \begin{cases}
3 \pmod{4}, & \textrm{if $p\equiv 3\pmod{4}$ and $k$ is odd}\\
1 \pmod{4}, & \textrm{otherwise}\\
\end{cases}$$ 
It follows that if $p\equiv 3\pmod{4}$ and $k$ is odd then $\frac{q-1}{2}$ is an odd integer and therefore $2s\equiv 1\pmod{\frac{q-1}{2}}$ has a solution. Therefore if $p\equiv 3\pmod{4}$ and $k$ is odd then $\mathbb{F}_{p^k}$ is a definite field.  Suppose now that  $k$ is even or $p\equiv 1\pmod{4}$.   Since $q=p^k\equiv 1\pmod{4}$ then $\fq^*$ has order $q-1$ which is divisible by $4$ so some element of $\fq^*$ has order $4$.  That is, there are elements $b,c\in \fq\setminus\{1\}$ such that $c^2=b$ and $b^2=1$.  Hence $b\neq 1\in \mathcal{P}$.  Define $f:\mathcal{P}\rightarrow \mathcal{P}$ by $f(x)=x^2$.
Since $\fq$ is finite but $f(1)=f(b)$, it follows that $f$ is not an onto map and therefore for some $y\in \mathcal{P}$ there is no $\mu\in \mathcal{P}$ such that $\mu^2=y$ and therefore $y$ has no positive square roots.
\end{proof}

\begin{definition}
A symmetric matrix, $A$, over a finite definite field $\fq$ is said to have a \textbf{Cholesky decomposition} if $A=LL^T$ for some lower triangular matrix $L \in  M_n(\fq)$ where $L$ has positive elements along its diagonal.
\end{definition}
We saw in Proposition \ref{stack} 
that a very common definition for positive definiteness for Hermitian matrices does not translate to finite fields, however
the following definition does, so it will be used throughout the paper as the main definition of positive-definite matrices over a definite field.

\begin{definition}
If $A$ is a symmetric $n\times n$ matrix over a definite field, $A$ is \textbf{positive
definite} if it possesses a Cholesky decomposition.
\end{definition}
In the following theorem and throughout this paper, $M_n(\mathbb{F})$ refers to the set (or space) of $n\times n$ matrices with entries in $\mathbb{F}$. The following results are finite-field adaptations of standard results from elementary number theory \cite{hardy1979introduction}  and linear algebra \cite{johnson1985matrix}. For the curious reader, we have included brief proofs in the appendix.

\begin{theorem}\label{apx1}
If $A \in M_n(\fq)$ and $A = LL^T$ for some lower triangular matrix
$L \in M_n(\fq)$ whose diagonal elements are all nonzero, then the leading principal
minors of $A$ are positive.   
\end{theorem}
\begin{lemma}\label{apx2}
If $A$ is a symmetric matrix over a definite field with an $LDU$ decomposition
where $L$ and $U$ have all ones along their diagonals and the diagonal entries of $D$ are
positive, then $A$ has a Cholesky decomposition.
\end{lemma}

\begin{lemma}\label{pivots}\label{apx3}  
 Let $A_k$ be the $k \times k$ leading principal submatrix of an $n \times n$ matrix $A$. If $A$ has an $LDU$ factorization, $A = LDU$, where $L$ is a lower triangular matrix with all ones along its diagonal, $U$ is upper triangular with all ones along its diagonal, and $D$ is diagonal, then $\det(A_k) = d_{11}d_{22} \cdots  d_{kk}$. The 1st pivot is $d_{11} = \det(A_1) = a_{11}$ and the $k$th pivot for $k = 2, 3, \cdots, n$ is $d_{kk} = \det(A_k)/\det(A_{k-1})$, where $d_{kk}$ is the $(k, k)$-th entry of $D$ for all $k = 1, 2, \cdots , n$.
\end{lemma}

\begin{lemma}\label{LDUChol}\label{apx4}  
If $A$ is a symmetric matrix over a definite field with an $LDU$ decomposition where $L$ and $U$ have all ones along their diagonals and the entries of $D$ are positive, then $A$ has a Cholesky decomposition. 
\end{lemma}

\begin{theorem}[Corollary 3.5.5 of \cite{johnson1985matrix}]
If $A$ is invertible, then it admits an $LDU$ factorization if and only if all its leading principal minors are nonsingular.
\end{theorem}

\begin{corollary}\label{apx5}
If all leading principal minors of a symmetric matrix $A$ over a definite field are positive, then $A$ has a Cholesky decomposition.
\end{corollary}

\begin{lemma}\label{Gram}\label{apx6}
All leading principal submatrices of a Gram matrix are also Gram matrices.
\end{lemma}

\begin{theorem}\label{apx7}
A matrix, $M\in \mathcal{M}_n(\F_q)$, is a Gram matrix if and only if it is positive definite. 
\end{theorem}

\section{Counterexamples}

There are some properties of positive-definiteness, however, which no longer hold over definite fields.  Since real matrices are covered by the Hermitian case, we turn our attention specifically to \textbf{finite} definite fields. In this section, we consider the a number of classical equivalences which do not hold and present counterexamples. We also take a look into some of the other properties of Hermitian positive definite matrices and provide counterexamples to show they cannot hold over finite fields.

\begin{theorem}\label{HermImplies}   
If $A$ is a positive definite Hermitian matrix, that is, over $\mathbb{R}$, or $\mathbb{C}$, the following hold:
\begin{enumerate}
    \item $A$ has positive eigenvalues.
    \item The associated sesquilinear form is an inner product.
    \item All principal submatrices of $A$ are positive definite.
    \item $A^{-1}$ is positive definite.
    \item If $B$ is a positive definite Hermitian matrix, then $A + B$ is positive definite.
    \item If $B$ is a positive definite Hermitian matrix, then $ABA$ and $BAB$ are positive definite.
    \item If $B$ is a positive definite Hermitian matrix, then the Hadamard product $A \circ B$ is positive definite and the Frobenius inner product, $A:B$ is positive.
\end{enumerate}
\end{theorem}

\begin{theorem} 
The properties of Theorem \ref{HermImplies} do not hold in general over finite definite fields.
\end{theorem}

\begin{proof}
We provide at least one counter example from a definite field for each property or explain why the described property does not hold.
\begin{enumerate}
\item The following matrix, in $\mathcal{M}_2(\F_7)$, is positive definite as all leading principal minors are positive in $\F_7$ but has eigenvalues of 6 and 5, which are not positive in the field. 
$$\begin{bmatrix}
    2 & 4 \\
    4  & 2 
\end{bmatrix}$$
For another example, consider the following in $M_3(\F_3)$, which has eigenvalues 1,2,2. $$\begin{bmatrix}
    1 & 0 & 2 \\
    0 & 1 & 1 \\
    2 & 1 & 0
\end{bmatrix}$$
One might hope that the converse still holds, that positive eigenvalues always indicate that a matrix is positive definite, but this sadly is also untrue. The following matrix over $\F_7$ has eigenvalues of $1$ and $2$, which are positive in $\F_7$, but not all leading principal minors are positive for the matrix, thus it is not positive definite. 
$$\begin{bmatrix}
6 & 6 \\
6 & 4
\end{bmatrix}
$$
\item The sesquilinear form defined by a matrix $A$ is a function from $\F_{q^2}^n \rightarrow \F_{q^2}^n$ given by $\langle x,y\rangle = y^TAx$ for $x, y \in \F_q$. For this to be an inner product, we must have that $\langle x,x\rangle$ is nonzero and positive for all nonzero $x$. However, in finite fields this form is isotropic, as seen in Proposition \ref{stack}, and therefore can be zero for nonzero $x$. 
\item The following matrix, in $\mathcal{M}_3(\F_3)$, is positive definite: $$\begin{bmatrix}
    1 & 2 & 0\\
    2 & 2 & 0 \\
    0 & 0 & 1
\end{bmatrix}$$  however it contains principal submatrix $\begin{bmatrix}
2 & 0\\
0 & 1
\end{bmatrix}$ which is not positive definite (evaluate the determinant).\\
For another example, consider the following matrix in $\mathcal{M}_3(\F_2)$, which is positive definite:
$$\begin{bmatrix}
    1 & 1 & 0\\
    1 & 0 & 0 \\
    0 & 0 & 1
\end{bmatrix}$$  One principal submatrix is $\begin{bmatrix}
0 & 0\\
0 & 1
\end{bmatrix}$, which is not positive definite (evaluate the determinant).  In general, there are many positive definite matrices with elements along the diagonal which are not positive. Taking a principal submatrix that causes one of these elements to be in the upper left corner will produce a submatrix that is not positive definite.

\item The following matrix, $\mathcal{M}_3(\F_3)$, is positive definite:
$$\begin{bmatrix}
    1 & 2 & 0\\
    2 & 2 & 0\\
    0 & 0 & 1
\end{bmatrix}$$
However, we have that $A^{-1}$ is $$\begin{bmatrix}
    2 & 1 & 0\\
    1 & 1 & 0\\
    0 & 0 & 1
\end{bmatrix}$$
which is not positive definite, since in particular $A_1 = [2]$ does not have positive determinant.\\
For another example, consider the following matrix, in $\mathcal{M}_3(\F_2)$, which is positive definite:
$$\begin{bmatrix}
    1 & 1 & 1\\
    1 & 0 & 0\\
    1 & 0 & 1
\end{bmatrix}$$
However, 
$$A^{-1}=\begin{bmatrix}
    0 & 1 & 0\\
    1 & 0 & 1\\
    0 & 1 & 1
\end{bmatrix}$$ which is not positive definite, as $A_1 = [0]$ does not have a positive determinant.

\item Consider the following in $\mathcal{M}_3(\F_2)$:
$$\begin{bmatrix}
    1 & 0 & 0\\
    0 & 1 & 0\\
    0 & 0 & 1
\end{bmatrix} + \begin{bmatrix}
    1 & 0 & 0\\
    0 & 1 & 0\\
    0 & 0 & 1
\end{bmatrix} = \begin{bmatrix}
    0 & 0 & 0\\
    0 & 0 & 0\\
    0 & 0 & 0
\end{bmatrix}$$
The identity matrix is positive definite, yet the zeros matrix is obviously not.
\item Consider the following positive definite matrices in $\mathcal{M}_2(\F_7)$:
$$A = \begin{bmatrix}
    2 & 1\\
    1 & 5
\end{bmatrix}, B = \begin{bmatrix}
    4 & 3\\
    3 & 6
\end{bmatrix}$$
We have that $ABA$ is $$\begin{bmatrix}
    6 & 1 \\
    1 & 2
\end{bmatrix}$$
This matrix is not positive definite, since $(ABA)_1 = [6]$ does not have positive determinant.\\
For another example, consider the following matrices, in $\mathcal{M}_3(\F_2)$, which are positive definite:
$$ A = 
\begin{bmatrix}
    1 & 0 & 1\\
    0 & 1 & 0\\
    1 & 0 & 0
\end{bmatrix}
, B= \begin{bmatrix}
    1 & 0 & 1\\
    0 & 1 & 1\\
    1 & 1 & 1
\end{bmatrix}$$ However
$$ABA=\begin{bmatrix}
    0 & 1 & 0\\
    1 & 1 & 0\\
    0 & 0 & 1
\end{bmatrix}$$
which is not positive definite since its leading $1\times 1$ principal minor is $0$. 

\item Consider $\begin{bmatrix}
    1 & 4 \\
    4 & 3
\end{bmatrix}$ and $\begin{bmatrix}
    2 & 2 \\
    2 & 3
\end{bmatrix}$ in $\mathcal{M}_2(\F_7)$, which are both positive definite. However, their Hadamard product is $\begin{bmatrix}
    2 & 1 \\
    1 & 2
\end{bmatrix}$
whose determinant is 3, which is not positive in $\F_7$ and therefore the matrix is not positive definite. Considering this same pair of matrices, their Frobenius inner product is 6, which is also not positive.\\
For another example, consider the following matrices, in $\mathcal{M}_3(\F_2)$, which are positive definite:
$$A= \begin{bmatrix}
    1 & 1 & 0\\
    1 & 0 & 1\\
    0 & 1 & 0
\end{bmatrix}, B= \begin{bmatrix}
    1 & 0 & 1\\
    0 & 1 & 1\\
    1 & 1 & 1
\end{bmatrix}. $$ However, their Hadamard product is
$$H= \begin{bmatrix}
    1 & 0 & 0\\
    0 & 0 & 1\\
    0 & 1 & 0
\end{bmatrix}$$ which is not positive definite, as $H_2= \begin{bmatrix}
    1 & 0 \\
    0 & 0 
\end{bmatrix}$ does not have a positive determinant.
For one more example, consider the following matrices, in $\mathcal{M}_3(\F_2)$, which are positive definite:
$$A= \begin{bmatrix}
    1 & 0 & 0\\
    0 & 1 & 1\\
    0 & 1 & 0
\end{bmatrix}, B= \begin{bmatrix}
    1 & 0 & 1\\
    0 & 1 & 0\\
    1 & 0 & 0
\end{bmatrix}.$$ Their Frobenius inner product is $0$, which is not positive. Observe that we need only find any pair of positive definite matrices where there is an even number of entries satisfying $a_{ij} = b_{ij} = 1$ and their Frobenius inner product to give a non-positive result.  
\end{enumerate}

\end{proof}

\section{Other Properties}

Some of the properties that Hermitian positive definite matrices possess do, however, analogize over definite fields.  

\begin{theorem}
If $A$ is a positive definite matrix over a definite field $\F$, and $r$ is positive in $\F$, then $rA$ is also positive definite.
\end{theorem}

\begin{proof}
If $A$ is a $n \times n$ positive definite matrix over a definite field $\F$, then it possesses a Cholesky decomposition, $A = LL^T$. If $\det(L) = \mu$ and $r$ is a square in $\F$, that is $r = s^2$ for $s \in \F$, then $\det(rA) = \det(rLL^T) = \det(rL)\det(L^T) = r^n\mu\mu = s^{2n}\mu^2 = (s^n\mu)^2$. As all leading principal submatrices have a similar decomposition, all leading principal minors of $rA$ are positive by a similar argument and thus $rA$ is positive definite.
\end{proof}

In the last section, we saw that inverses of positive definite matrices over finite definite fields are not positive definite. It is true, however, that the inverse matrix conjugated by the ``exchange'' or ``anti-diagonal identity'' matrix $\nabla$ is positive definite.

\begin{definition}
For an invertible matrix $A$, define its \textbf{anti-inverse} as $\nabla A^{-1} \nabla$ where $\nabla$ is the exchange matrix, with ones along its antidiagonal and zeroes elsewhere. That is, in the $n \times n$ case of $\nabla, a_{ij} = 1$ if $i + j = n+1$ and $a_{ij}= 0$ otherwise. For example, the $3 \times 3$ case of $\nabla$ is $$\begin{bmatrix}
 0 & 0 & 1 \\
 0 & 1 & 0 \\
 1 & 0 & 0 
\end{bmatrix}$$
\end{definition}
The following lemma will be helpful in showing that the anti-inverse is positive definite. 

\begin{lemma}
Every principal submatrix of a lower triangular matrix is lower triangular. 
\end{lemma}

\begin{proof}
Let $L$ be a lower triangular matrix. Deleting the first column and row clearly produces a lower triangular matrix, and similarly if we delete the last row and column. Now, suppose we delete the $i$th row and column. We have 
$$\begin{bmatrix}
    l_{11}  & \cdots & 0 & \cdots & 0\\
    \vdots &  \ddots & \vdots &  & \vdots \\
    l_{i1} &  \cdots & l_{ii} & \cdots & 0\\
    \vdots &    & \vdots & \ddots & \vdots \\
    l_{n1}  & \cdots & l_{ni} & \cdots & l_{nn}\\
\end{bmatrix} \rightarrow 
\begin{bmatrix}
    l_{11}  & \cdots & 0 & 0 &  \cdots & 0\\
    \vdots  & \ddots & \vdots & \vdots &  & \vdots \\
    l_{(i-1)1}  & \cdots & l_{(i-1)(i-1)} & 0 & \cdots & 0\\
    l_{(i+1)1} &  \cdots & l_{(i+1)(i-1)} & l_{(i+1)(i+1)} & \cdots & 0\\
    \vdots &    & \vdots & \vdots &  \ddots &\vdots  \\
    l_{n1} &  \cdots & l_{n(i-1)} & l_{n(i+1)} & \cdots & l_{nn}\\
\end{bmatrix}$$

The $(i-1)$-th leading principal submatrix is still lower triangular. As both row and column $i$ are removed, the original $l_{(i+1)(i-1)}$ entry now becomes the entry $l'_{ii}$ in the $i$th row and $i$th column of the new matrix. The rest of the matrix is shifted and retains the form of a lower triangular matrix.
\end{proof}

\begin{theorem}
If $A$ is a positive definite matrix in a definite field $\F$, then its anti-inverse is also positive definite. 
\end{theorem}
\begin{proof}
Let $A$ be a positive definite matrix in a definite field $\F$. It is invertible and thus we can consider its anti-inverse. As $A$ is positive definite, we have $A = LL^T$ as a Cholesky decomposition. Note that $\nabla \nabla = I$ and therefore we have, writing $L^{-T} = (L^{-1})^T$,
\begin{align*}
    A^{-1} &= L^{-T} L^{-1}\\
    \nabla A^{-1} \nabla &= \nabla L^{-T} L^{-1} \nabla\\
    &= \nabla L^{-T} (\nabla \nabla) L^{-1} \nabla\\
    &= (\nabla L^{-T} \nabla) (\nabla L^{-1} \nabla)
\end{align*}

Note that right multiplying by $\nabla$ reverses the columns of the matrix and left multiplication by $\nabla$ reverses the rows. Thus, $\nabla L^{-T} \nabla$ takes an upper triangular matrix, $L^{-T}$, to a lower triangular matrix and $\nabla L^{-1} \nabla$ takes a lower triangular matrix to an upper triangular matrix. In fact, we have $$(\nabla L^{-T} \nabla)^T = (\nabla L^{-1} \nabla)$$

Thus, $\nabla A^{-1} \nabla$ takes the correct form to have a Cholesky decomposition. We need only check the diagonal elements of $\nabla L^{-T} \nabla$ are positive. As both the rows and columns are reversed by conjugating by $\nabla$, the diagonal elements of $L^{-T}$ are still the diagonal elements of $\nabla L^{-T} \nabla$, simply in a different order. As $LL^T$ is a Cholesky decomposition of $A$, the diagonal elements of $L$ are positive, and we need only check that the diagonal elements of $L^{-1}$ are positive.

When taking the inverse of $L$, the $(i,i)$ entry will be $1/\det(L)$ multiplied by the principal minor of the submatrix created by deleting the $i$th row and $i$th column. As this submatrix will be lower triangular and have diagonal elements which are a subset of those from $L$, the principal minor will be positive. Thus, the $i$th diagonal element of $L^{-1}$ is positive. As the diagonal elements of $L^{-1}$ are positive, so are those of $L^{-T}$. Thus, $(\nabla L^{-T} \nabla) (\nabla L^{-1} \nabla)$ is a Cholesky decomposition for $\nabla A^{-1}\nabla$. 
\end{proof}

In the last section, we provided counterexamples that proved the Hadamard product and the Frobenius product need not be positive definite nor positive respectively. It is true, however, that the Kronecker product of two positive definite matrices, even for definite fields, is positive definite.

\begin{theorem}
If $A$ and $B$ are positive definite matrices in a definite field $\F$, then so is their Kronecker product. In fact, if $A = LL^T$ and $B = MM^T$ then $A \otimes B = (L \otimes M)(L \otimes M)^T$
\end{theorem}
\begin{proof}
Let $A$ and $B$ be $n \times n$ positive definite matrices in a definite field $\F$, with $A = LL^T$ and $B = MM^T$ their Cholesky decompositions.
$$
L \otimes M = \begin{bmatrix}
    l_{11}M & 0 & \cdots & 0\\
    l_{21}M & l_{22}M & \cdots & 0\\
    \vdots & \vdots & \ddots & \vdots \\
    l_{k1}M & l_{k2}M & \cdots & l_{kk}M\\
\end{bmatrix},$$  $$(L \otimes M)^T = \begin{bmatrix}
    l_{11}M^T & l_{21}M^T & \cdots & l_{k1}M^T\\
    0 & l_{22}M^T & \cdots & l_{k2}M^T\\
    \vdots & \vdots & \ddots & \vdots \\
    0 & 0 & \cdots & l_{kk}M^T\\
\end{bmatrix}$$

Consider $(L \otimes M)(L \otimes M)^T$. When calculating any entry of this product, we will have a sum of scalars each multiplied by $MM^T$, and so we may factor this out by the distributivity of matrices over scalars. The sum of scalars, if considering the $(i,j)$th entry, is produced from the dot product of the $i$th row of $L$ with the $j$th column of $L^T$, exactly $a_{ij}$, the entry of $a$ in the $i$th row and $j$th column. Thus, the $(i,j)$th entry of $(L \otimes M)(L \otimes M)^T$ is $a_{ij}MM^T = a_{ij}B$ and thus $(L \otimes M)(L \otimes M)^T = A \otimes B$.

We need only check that the diagonal of $(L \otimes M)$ is positive. The diagonal elements of $L \otimes M$, as $M$ is lower triangular, are comprised of the diagonal elements of $L$ multiplied by the diagonal elements of $M$. As the diagonal elements of both $L$ and $M$ are positive, their product will also be positive. Thus, $A \otimes B$ has a Cholesky decomposition and is therefore positive definite. 
\end{proof}

\section{Pressing Sequences}\label{pressingsection}

\begin{definition}
Let a \textbf{$\F$-pseudograph}, for some field $\F$, be a graph $G= (V,f)$ with $V$ the set of vertices and $f: V \times V \rightarrow \F$ a function assigning a weight to each edge. That is, $f(x,y) = c$ assigns a weight of $c \in \F$ to the edge $xy$, and $f(x,x)$ is a weight assigned to the vertex $x \in V$. Each edge has only one associated weight, i.e., $f(x,y) = f(y,x)$. Note every pair of distinct vertices admits an edge, though some may simply have weight $0$ and sometimes $G$ is then identified with the restriction of $f$ to its support.
\end{definition}

For a vertex, we may refer to the vertex by its weight if the vertex label is understood. That is, if there is only one vertex of weight $d$, it may be referred to as vertex $d$. If there are more than one vertex with weight $d$, it will be referred to as vertex $v$ with weight $d$. 

\begin{definition}
The \textbf{weighted adjacency matrix} of a $\F$-pseudograph $G$ is $A(G)$ defined in the following way. Let $v_1,...,v_n$ be the vertices of $G$.
$$a_{ij} = f(v_i,v_j)$$
Clearly, the resulting matrix is symmetric.
\end{definition}

\begin{definition}
Consider a $\F-$pseudograph $G = (V,f)$. For a vertex $v$, with $f(v,v)$ positive in $\F$, \textbf{pressing} $v$ is the process of taking $G$ to $G' = (V, g)$ with $$g(x,y) = f(x,y) - \dfrac{f(x,v)f(y,v)}{f(v,v)}$$
\end{definition}
Note that such a press will also yield a symmetric weighted adjacency matrix $A(G')$.  The following figure illustrates a general press on the vertex with weight $a$:

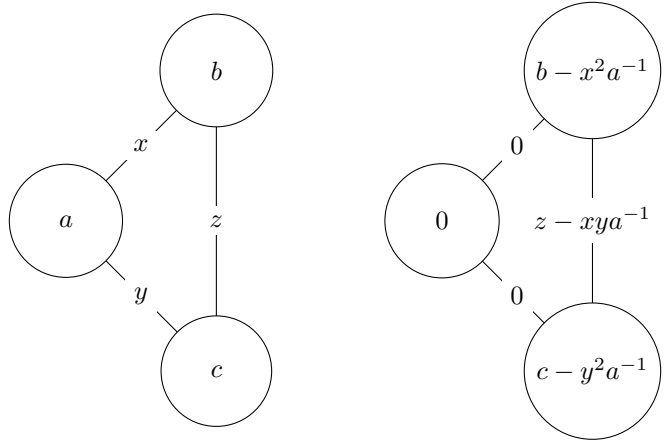
\begin{figure}[!ht]
        \begin{center}
\begin{tikzpicture}
\draw (0,0) node[circle,draw, fill=white]{\phantom{$xx$} $a$ \phantom{$xx$}} -- (2,2) node[circle,draw, fill=white]{\phantom{$xx$} $b$ \phantom{$xx$}} -- (2,-2) node[circle,draw, fill=white]{\phantom{$xx$} $c$ \phantom{$xx$}} -- (0,0);
\draw (5,0) node[circle,draw, fill=white]{\phantom{$xx$} $0$ \phantom{$xx$}} -- (7,2) node[circle,draw, fill=white]{$b-x^2a^{-1}$} -- (7,-2) node[circle,draw, fill=white]{$c-y^2a^{-1}$} -- (5,0);
\draw (1,1) node[fill=white]{$x$};
\draw (2,0) node[fill=white]{$z$};
\draw (1,-1) node[fill=white]{$y$};
\draw (6,1) node[fill=white]{$0$};
\draw (7,0) node[fill=white]{$z-xya^{-1}$};
\draw (6,-1) node[fill=white]{$0$};
\end{tikzpicture}
\end{center}

        \centering
        \caption{The weighted vertex $a$ is pressed to transform $G$ (left) to $G'$ (right)}
\end{figure}

The weighted adjacency matrices for the graphs in the figure as follows:
$$A(G) = \begin{bmatrix}
a & x & y\\
x & b & z\\
y & z & c
\end{bmatrix} \hspace{.75in} A(G') = \begin{bmatrix}
0 & 0 & 0\\
0 & b-\dfrac{x^2}{a} & z-\dfrac{xy}{a}\\
0 & z-\dfrac{xy}{a} & c- \dfrac{y^2}{a}\end{bmatrix}$$

Pressing in this fashion has the same effect as Gaussian elimination, except without row swaps, and such that the rows corresponding to pressed vertices are self-eliminated. A ``successful pressing sequence'' exists if we can complete this elimination to result in the all-zeroes matrix, which corresponds to the edgeless graph with vertices of weight 0.

\begin{theorem}
For a $\F-$pseudograph $G = (V,f)$, the vertices of $G$ in the usual order form a successful pressing sequence if and only if $A(G)$ is positive definite.
\end{theorem}
\begin{proof}
Let $G = (V,f)$ be a $\F$-pseudograph, and $A(G)$ its weighted adjacency matrix.

Suppose the vertices of $G$ -- in the order they are presented as indices of $A(G)$ -- form a successful pressing sequence. Thus, we can perform ordinary Gaussian elimination with no row swaps via the same process but omitting self-elimination steps, resulting in an $LU$ decomposition.  Each such elimination step begins by multiplying $A(G)$ by an elementary matrix $E_1$ on the left; the elements on the diagonal of $E_1$ are $1$ as we are not changing the entry associated with the vertex pressed, and $E_1$ is also lower triangular. As $A(G)$ is symmetric, the elimination can proceed by performing the same operations on columns, represented by right multiplication by $E_1^T$. That is, after the elimination operations arising from a press without self-elimination, we have $A(G') = E_1A(G)E_1^T$. 

If $G$ has a successful pressing sequence, it has a sequence of such Gaussian elimination steps which result in a diagonal matrix. That is, $EA(G)E^T = D$ for $E$ a product of elementary matrices representing row operations where the diagonal entries of $E$ are 1 and $D$ a diagonal matrix whose entries are the weights of the vertices before being pressed. As we only press positively weighted vertices, $D$ has all positive entries. In fact, $A(G) = E^{-1}DE^{-T}$. As $E$ is lower diagonal, so is $E^{-1}$ and in a similar fashion, $E^{-T}$ is upper triangular. Thus, $A(G)$ has a positive LDU decomposition and by Lemma \ref{LDUChol}, $A(G)$ has a Cholesky decomposition and is therefore positive definite.\\

If $A(G)$ is positive definite, it has a Cholesky decomposition $A(G) = LL^T$. As the diagonal entries of $L$ are positive, we have $A(G) = L'D(L')^T$ where $L'$ has all ones along its diagonal and $D$ is a diagonal matrix with all positive entries. So $L'^{-1}A(G)L'^{-T} = D$ and $A(G)$ can be row- and column-reduced without swaps to a diagonal matrix with all positive entries. Thus $G$ has a Gaussian elimination sequence without row swaps resulting in a positive diagonal matrix and also has a successful pressing sequence by pressing the vertices in the same order.
\end{proof}

\section{Future Work}

Over $\mathbb{F}_2$ a graph has a successful pressing sequence if and only if each component contains at least one non-white vertex (\cite{CooperDavis}). This characterization does not work over other fields (for example the $\mathbb{F}_3$ pseudo-graph with three vertices, each of weight $1$, and three edges, each of weight $2$, is not pressable in any order).  Furthermore, while it is a necessary condition that each component contain a positive vertex, this characterization does not suffice.  Thus, we ask the following question: Is there a polynomial-time algorithm to check whether a given pseudo-graph over a finite field is pressable {\it in some order}?

We have discussed positive definite matrices over definite fields and described many equivalences which can be analogized from the Hermitian case, and others which cannot. It would be perhaps be interesting also to consider semi-definiteness or negative definiteness over definite fields. 

In particular, we briefly considered but never truly investigated the possibility of using the Frobenius endomorphism to define some kind of Hermitian-like structure on these matrices. That is, instead of a conjugate transpose, what would happen if we applied the Frobenius map to every element of the matrix and then take the transpose? This notion may still cause a problem in the positive definite case, as we can simply take a vector in the base field, for which the ``Frobenius transpose'' would simply be the transpose, and could still produce $x^TAx = 0$ for some nonzero vector $x$. In terms of positive semi-definiteness, however, we wonder if this can be remedied and prove interesting.  We would be especially keen to restore the important role of positive eigenvalues in the theory of positive definiteness, perhaps by redefining ``positive'' via Frobenius endomorphisms.

It would also be interesting to consider whether nondefinite fields have some semblance of a positive definite structure given the right definitions. Furthermore, we identified counterexamples above to show that some classical properties do not hold over finite definite fields.  Could there be, however, a subset of finite definite fields for which some of these properties still indeed hold? For instance, can the positive eigenvalue equivalence be salvaged over a certain subset of definite fields?

We also wonder what applications the above framework might find. As positive definite matrices are used often in optimization problems, does this notion of positive definite over definite fields induce an analogue of geometric convexity over other fields besides $\mathbb{R}$ and $\mathbb{C}$? Can we solve optimization problems over finite fields?

\section{Appendix}

\noindent Proof of Theorem \ref{apx1}
\begin{proof}
Let $A \in M_n(\fq)$ such that $A=LL^T$ for some lower triangular matrix $L \in M_n(\fq)$ whose diagonal entries are nonzero.  Let $L_i$ denote the $i^{\textrm{th}}$ leading principal submatrix of $L$.   For each $1\leq i\leq n$, let $\mu_i=\textrm{det}(L_i)$ and observe that $\mu_i\neq 0\in \fq$.  Then the $i^{\textrm{th}}$ leading principal submatrix of $A$ will have decomposition $A_i=L_iL_i^T$ and therefore $\textrm{det}(A_i)=\textrm{det}(L_i)\textrm{det}(L_i^t)=\mu^2\neq 0$.
\end{proof}

\noindent Proof of Lemma \ref{apx2}
\begin{proof}
Let $A$ be a symmetric matrix over a definite field with an $LDU$ decomposition
such that all entries of $D$ are positive and all diagonal entries of $L$ and $U$ are 1. The
symmetry of $A$ and the uniqueness of the $LDU$ decomposition will yield $U = L^T$.
As the diagonal entries of $D$ are positive, 
$\sqrt{D}$ can be defined, and we construct it in the
following way:
Set $\sqrt{D}=\textrm{diag}(s_{1,1}, \ldots, s_{n,n})$ where $s_{i,i}$ is the positive square root of the element in the $i^\textrm{th}$ row and column of $D$.  Since both $L$ and $\sqrt{D}$ have zeros above the diagonal then $B=L\sqrt{D}$ is a lower triangular matrix and $A=BB^T$, as desired.
\end{proof}

\noindent Proof of Lemma \ref{apx3}
\begin{proof}
Let $A_k$ be the $k \times k$ leading principal submatrix of an $n \times n$ matrix $A$. Let $A$ have an $LDU$ factorization, $A = LDU$, where $L$ is a lower triangular matrix with all ones along its diagonal, $U$ is upper triangular with all ones along its diagonal, and $D$ is diagonal. Note that as $A$ has an $LU$ decomposition, all leading principal submatrices have full rank and thus all leading principal minors are nonzero.

Partition $A$ in the following way:
$$A = \begin{bmatrix}
    L_k & \textbf{0} \\
    L_{21}  & L_{22} 
\end{bmatrix}  \begin{bmatrix}
    D_k & \textbf{0} \\
    \textbf{0}  & D_{22} 
\end{bmatrix} \begin{bmatrix}
    U_k & U_{12} \\
    \textbf{0}  & U_{22} 
\end{bmatrix}$$
We thus have that $A_k$ can be written in the following manner:
$$A_k = L_kD_kU_k = \begin{bmatrix}
    L_{k-1} & \textbf{0} \\
    \textbf{d}  & 1 
\end{bmatrix}  \begin{bmatrix}
    D_{k-1} & \textbf{0} \\
    \textbf{0}  & d_{kk} 
\end{bmatrix} \begin{bmatrix}
    U_{k-1} & \textbf{c} \\
    \textbf{0}  & 1 
\end{bmatrix}$$
For $k = 1$, we have $A_1 = [1][d_{11}][1]$ and thus $\det(A_{1}) = d_{11} = a_{11}$. If the result holds for $\ell< k$, we have $\det(A_k) = \det(D_{k-1})d_{kk} = \det(A_{k-1})d_{kk} = d_{11}...d_{kk}$. The result follows as the pivots are exactly the entries of $D$.
\end{proof}

\noindent Proof of Lemma \ref{apx4}
\begin{proof}
Let $A$ be a symmetric matrix over a definite field with an $LDU$ decomposition such that all entries of $D$ are positive and all diagonal entries of $L$ and $U$ are 1. The symmetry of $A$ and the uniqueness of the $LDU$ decomposition will yield $U = L^T$. As the elements of $D$ are positive, $\sqrt{D}$ can be defined, and we construct it in the following way. If $r_{ii} = \sqrt{d_{ii}}$:
\[\sqrt{D} = \text{diag}(d'_{11} d'_{22} ... d'_{nn}) = 
\begin{cases}
d'_{ii} = r_{ii} \hspace{1cm} \text{ if } r_{ii} \text{ is positive}\\
d'_{ii} = -r_{ii} \hspace{1cm} \text{ otherwise}
\end{cases}
\]
Thus, $\sqrt{D}$ is a diagonal matrix with positive diagonal entries. Define $R = L\sqrt{D}$. As $L$ has a diagonal of all 1's, $R$ is a lower triangular matrix with positive diagonal entries and $A = RR^T$ as desired.
\end{proof}

\noindent Proof of Corollary \ref{apx5}
\begin{proof}
Let $A$ be a symmetric matrix in $M_n(\mathbb{F})$ for a definite field $F$ such that all
leading principal minors are positive. Thus, all leading principal submatrices have
full rank and $A$ is invertible. So, $A = LDU$
where $D$ is a diagonal matrix and $U$ and $L$ have all ones on their diagonal. The
symmetry of $A$ and the uniqueness of the $LDU$ decomposition will yield that $U = L^T$ . 
As all leading principal minors are positive, the pivots of $A$, found by the process
described in Lemma \ref{pivots}, are positive and are, in fact, the diagonal entries of $D$. Thus, as we
have an LDU decomposition where all the diagonal elements of $D$ are positive, by Lemma \ref{LDUChol},
we can define $R = L\sqrt{D}$, a lower triangular matrix with positive diagonal entries,
and $A = RR^T$ as desired.
\end{proof}

\noindent Proof of Lemma \ref{apx6}
\begin{proof}
Let $G$ be a Gram matrix of vectors $v_1, v_2, ..., v_n$. That is, 
$$G = \begin{bmatrix}
    \langle v_1, v_1 \rangle & \langle v_1, v_2 \rangle & \cdots & \langle v_1, v_n \rangle\\
    \langle v_2, v_1 \rangle & \langle v_2, v_2 \rangle & \cdots & \langle v_2, v_n \rangle\\
    \hdots & \hdots & \ddots & \hdots\\
    \langle v_n, v_1 \rangle & \langle v_n, v_2 \rangle & \cdots & \langle v_n, v_n \rangle
\end{bmatrix}$$

Any leading principal submatrix, $G_k$ will take the form
$$G_k = \begin{bmatrix}
    \langle v_1, v_1 \rangle & \langle v_1, v_2 \rangle & \cdots & \langle v_1, v_k \rangle\\
    \langle v_2, v_1 \rangle & \langle v_2, v_2 \rangle & \cdots & \langle v_2, v_k \rangle\\
    \hdots & \hdots & \ddots & \hdots\\
    \langle v_k, v_1 \rangle & \langle v_k, v_2 \rangle & \cdots & \langle v_k, v_k \rangle
\end{bmatrix}$$
Thus, $G_k$ is a Gram matrix on the vectors $v_1, v_2, ..., v_k$ as these vectors are still linearly independent.
\end{proof}

\noindent Proof of Theorem \ref{apx7}

\begin{proof}
Let $M \in \mathcal{M}_n(\F_q).$

Suppose $M$ is a Gram matrix. Thus, $M = A^TA$ where the columns of $A$ are $x_1, x_2, ..., x_n \in \F_q^n$, which are linearly independent. Now, $M_k$ will be equivalent to $A_k^TA_k$ where $A_k$ has columns $x_1,...,x_k$ by Lemma \ref{Gram}. We have $$\det(M_k) = \det(A_k^TA_k) = \det(A_k)^2$$ for $\det(A_k)^2 \in \F_q$. As all leading principal minors are positive, $A$ is positive definite.

Now suppose that $M$ is a positive definite matrix. Thus, $M = LL^T$ with the columns of $L$ denoted by $l_1, l_2, ..., l_n$. As $M$ is invertible, so is $L$ and thus these $l_i$ are linearly independent. $M$ is therefore a Gram matrix for the vectors $l_1, l_2, ..., l_n.$
\end{proof}
\vfill

\end{document}